\documentclass[11pt]{amsart}
\usepackage{amsfonts}
\usepackage{fullpage, amsmath, amsthm,amsfonts,amssymb,stmaryrd, mathrsfs,amscd}\usepackage{graphicx}
\usepackage{hyperref}

\makeatletter
\usepackage[pdftex]{color}
\usepackage[all]{xy}

\newtheorem{thm}{Theorem}
\newtheorem{prop}[thm]{Proposition}
\newtheorem{lem}[thm]{Lemma}
\newtheorem{lem-constr}[thm]{Lemma-Construction}

\newtheorem{lem-def}[thm]{Lemma-Definition}
\newtheorem{cor}[thm]{Corollary}

\theoremstyle{definition}

\theoremstyle{definition}

\newtheorem{ex}[thm]{Example}

\newtheorem{rmk}[thm]{Remark}
\theoremstyle{definition}

\numberwithin{equation}{section}

\input xy
\xyoption{all}

\newcommand{\quash}[1]{}  
\newcommand{\nc}{\newcommand}
\nc{\on}{\operatorname}

\DeclareFontEncoding{OT2}{}{} 

\newcommand{\frakg}{{\mathfrak g}}

\newcommand{\frakl}{{\mathfrak l}}

\newcommand{\frakp}{{\mathfrak p}}

\newcommand{\fraks}{{\mathfrak s}}

\newcommand{\bA}{{\mathbb A}}

\newcommand{\bC}{{\mathbb C}}

\newcommand{\bF}{{\mathbb F}}
\newcommand{\bG}{{\mathbb G}}

\newcommand{\bM}{{\mathbb M}}

\newcommand{\bQ}{{\mathbb Q}}

\newcommand{\bV}{{\mathbb V}}

\newcommand{\bZ}{{\mathbb Z}}

\newcommand{\mE}{{\mathcal E}}
\newcommand{\mF}{{\mathcal F}}

\newcommand{\mH}{{\mathcal H}}

\newcommand{\mO}{{\mathcal O}}

\nc{\Tate}{\mathrm{Tate}}

\nc{\al}{{\alpha}} \nc{\be}{{\beta}} \nc{\ga}{{\gamma}}
\nc{\ve}{{\varepsilon}} \nc{\Ga}{{\Gamma}} \nc{\la}{{\lambda}}
\nc{\La}{{\Lambda}}

\nc{\ad}{{\on{ad}}}
\nc{\Ad}{{\on{Ad}}}
\nc{\aff}{{\on{aff}}}
\nc{\Aff}{{\mathbf{Aff}}}
\nc{\Aut}{{\on{Aut}}}
\nc{\Bun}{{\on{Bun}}}
\nc{\cha}{{\on{char}}}
\nc{\Ch}{{\on{Ch}}}
\nc{\cl}{{\on{cl}}}
\nc{\Coh}{{\on{Coh}}}
\nc{\der}{{\on{der}}}
\nc{\Der}{{\on{Der}}}
\nc{\DerAff}{{\on{DerAff}}}
\nc{\diag}{{\on{diag}}}
\nc{\dR}{{\on{dR}}}
\nc{\End}{{\on{End}}}
\nc{\Fl}{{\mF\ell}}
\nc{\Gal}{{\on{Gal}}}
\nc{\Gr}{{\on{Gr}}}
\nc{\Hom}{{\on{Hom}}}
\nc{\id}{{\on{id}}}
\nc{\Id}{{\on{Id}}}
\nc{\ind}{{\on{ind}}}
\nc{\Ind}{{\on{Ind}}}
\nc{\infGrp}{{\infty\on{-Groupoid}}}
\nc{\inv}{{\on{Inv}}}
\nc{\Iso}{{\on{Isom}}}
\nc{\Lie}{{\on{Lie}}}
\nc{\Loc}{{\on{Loc}}}
\nc{\Nm}{{\on{Nm}}}
\nc{\Pic}{{\on{Pic}}}
\nc{\pf}{{\on{pf}}}
\nc{\pr}{{\on{pr}}}
\nc{\rec}{{\on{rec}}}
\nc{\res}{{\on{res}}}
\nc{\Iw}{\on{Iw}}
\newcommand{\Res}{{\on{Res}}}
\nc{\Sat}{{\on{Sat}}}

\nc{\Sh}{{\on{Sh}}}
\nc{\Sht}{{\on{Sht}}}
\nc{\Shv}{{\on{Shv}}}
\nc{\Spec}{{\on{Spec}}}
\nc{\Spf}{{\on{Spf}}}

\nc{\tr}{{\on{tr}}}

\newcommand{\Mod}{{\mathrm{-Mod}}}

\newcommand{\GL}{{\on{GL}}}
\nc{\GSp}{{\on{GSp}}} \nc{\GU}{{\on{GU}}} \nc{\SL}{{\on{SL}}}
\nc{\SU}{{\on{SU}}} \nc{\SO}{{\on{SO}}}

\nc{\Ql}{{\overline{\bQ}_\ell}}

\nc{\fg}{\frakg}
\nc{\fp}{\frakp}

\nc{\rat}{\overline{\bQ}}
\nc{\triv}{{\bf{1}}}

\nc{\dm}{/\!\!/}

\def\xcoch{\mathbb{X}_\bullet}
\def\xch{\mathbb{X}^\bullet}

\nc{\wt}{\mathrm{wt}}

\nc{\wSh}{\widetilde{\Sht}}
\nc{\Sph}{\on{Sph}}
\nc{\Fr}{\on{Frob}}
\nc{\Fp}{{^\sigma\mE}}
\nc{\und}{\underline}
\nc{\mmu}{{\mu_\bullet}}
\nc{\nnu}{{\nu_\bullet}}
\nc{\Hk}{{\on{Hk}}}
\nc{\lhk}{\Hk^{\on{loc}}}
\newcommand{\IC}{\on{IC}}
\nc{\bb}{{\mathbf{b}}}
\nc{\bp}{{\mathbf{p}}}

\nc{\MV}{{\bM\bV}}
\nc{\FM}{{\on{FM}}}
\nc{\FS}{{\on{FS}}}
\nc{\Rep}{{\on{Rep}}}
\nc{\SGr}{{\on{SGr}}}
\nc{\Mon}{{\on{Mon}}}
\nc{\FFS}{{\on{FFS}}}
\nc{\FFSI}{{\on{FFS}^{+}}}
\nc{\FFM}{{\on{FFM}}}
\nc{\FFMI}{{\on{FFM}^{+}}}
\nc{\un}{{\on{unip}}}

\newcommand{\xdashrightarrow}[2][]{\ext@arrow 3359 \rightarrowfill@@{#1}{#2}}
\def\rightarrowfill@@{\arrowfill@@\relax\relbar\shortrightarrow}
\def\arrowfill@@#1#2#3#4{%
  $\m@th\thickmuskip0mu\medmuskip\thickmuskip\thinmuskip\thickmuskip
   \relax#4#1
   \xleaders\hbox{$#4#2$}\hfill
   #3$%
}

\title{A note on Integral Satake isomorphisms}
\author{Xinwen Zhu}
\address{California Institute of Technology, 1200 East California Boulevard, Pasadena, CA 91125, USA}
\email{xzhu@caltech.edu}
\thanks{Supported by the National Science Foundation under agreement Nos. DMS-1902239.}

\begin{document}
\maketitle
\begin{abstract}
We formulate a Satake isomorphism for the integral spherical Hecke algebra of an unramified $p$-adic group $G$ and generalize the formulation to give a description of the Hecke algebra $H_G(V)$ of weight $V$, where $V$ is a lattice in an irreducible algebraic representation of $G$.
\end{abstract}
\tableofcontents

In this note, we formulate a \emph{canonical} integral Satake isomorphism, by identifying the integral spherical Hecke algebra $H_G$ of an unramified $p$-adic group $G$
with a $\bZ$-algebra associated to an affine monoid $V_{\hat{G},\rho_\ad}$ of the Langlands dual group $\hat{G}$. The precise statement can be found in Proposition \ref{prop: main}. 
There are two motivations. First, the usual Satake isomorphism depends on a choice of $q^{1/2}$ and therefore only gives a description of $H_G\otimes\bZ[q^{\pm 1/2}]$ in terms of the dual group $\hat{G}$ (e.g. see \cite{Gr}). Using Deligne's modification of the Langlands dual group  (e.g. see \cite{BG}), one can formulate a canonical Satake isomorphism for $H_G\otimes\bZ[p^{-1}]$. On the other hand, there is the mod $p$ Satake isomorphism, as discussed in \cite{He, HV}. 
However, the Langland duality does not appear explicitly in these works. 
Therefore, it is desirable to extend the classical Satake isomorphism to $\bZ$, which after mod $p$ recovers the mod $p$ Satake isomorphism. Another motivation comes from the recent work of R. Cass (\cite{Ca}) on the geometric Satake equivalence for perverse $\bF_p$-sheaves on the affine Grassmannian. In his work, what controls the tensor category is not the Langlands dual group, but some affine monoid $M_G$. I hope $V_{\hat{G},\rho_\ad}$ and $M_G$ are closely related\footnote{However, as explained to me by Cass, the monoid appearing in his work is solvable and therefore cannot be isomorphic to the one appearing in this note.}.

Our formulation generalizes easily to give a description of the Hecke algebra $H_G(V)$ of weight $V$ in terms of the affine monoid $V_{\hat{G},\la_\ad+\rho_\ad}$, where $V$ is a lattice in an irreducible algebraic representation of $G$ of highest weight $\la$. See Proposition \ref{prop: main weights} for the precise formulation. It follows from our formulation that the ring of invariant functions on the Vinberg monoid specializes to all $H_G(V)$'s.

\section{The Satake isomorphism}
\subsection{The $C$-group}
Let $G$ be a connected reductive group over a field $F$. Let $(\hat G, \hat B, \hat T ,\hat{e})$ be a pinned dual group of $G$ over $\bZ$. There is a natural action of the Galois group $\Ga_F$ of $F$ on $(\hat G, \hat B, \hat T, \hat{e})$,  induced by its action on the root datum of $G$. This action factors as $\Ga_F\twoheadrightarrow\Ga_{\widetilde{F}/F}\stackrel{\xi}{\hookrightarrow} \Aut(\hat G, \hat B, \hat T, \hat e)$, where $\Ga_{\widetilde{F}/F}$ is the Galois group of a finite Galois extension $\widetilde F/F$.

Let $2\rho:\bG_m\to\hat{T}$ denote the cocharacter given by sum of positive coroots of $\hat{G}$ (with respect to $\hat B$), and $2\rho_\ad$ its projection to the adjoint torus $\hat{T}_{\ad}\subset \hat{G}_\ad=\hat{G}/Z_{\hat{G}}$ of $\hat{G}$. It admits a unique square root $\rho_\ad:\bG_m\to \hat{T}_\ad$.
Define an action of $\bG_m$ on $\hat G$ by
\[
   \Ad\rho_\ad: \bG_m\stackrel{\rho_\ad}{\to} \hat T_\ad\stackrel{\Ad}{\to} \on{Aut}(\hat G).
\]
Note that $\Ad\rho_\ad$ still preserves $(\hat B,\hat T)$, but not $\hat e$. The two actions $\Ad\rho_\ad$ and $\xi$ on $\hat{G}$ commute with each other. Following the terminology of Buzzard-Gee \cite{BG}, we define the $C$-group of $G$ as an affine reductive group scheme over $\bZ$ as
\[
   {}^CG:=\hat{G}\rtimes(\bG_m\times\Ga_{\widetilde{F}/F}).
\]
This group scheme appears naturally from the geometric Satake equivalence, as to be reviewed in Theorem \ref{T: geom Sat} below. 
Note that our definition is different from \cite[Definition 5.3.2]{BG}, but is equivalent to it (see below).
We may write 
\[
   {}^CG\cong {^L}G\rtimes \bG_m\cong \hat{G}^T\rtimes\Ga_{\widetilde{F}/F},
\]   
where ${^L}G:=\hat{G}\rtimes_\xi\Ga_{\widetilde{F}/F}$ is the usual Langlands dual group, and
where $\hat G^T:=\hat G\rtimes_{\Ad\rho_\ad} \bG_m$, which fits into the short exact sequence
\begin{equation}\label{hat GT}
   1\to \hat G\to \hat G^T\xrightarrow{d_{\rho_\ad}} \bG_m\to 1.
\end{equation}   
Note that there is an isomorphism
\begin{equation}\label{R: another version}
   (\hat{G}\times \bG_m)/(2\rho\times\id)(\mu_2)\cong \hat{G}^T,\quad (g,t)\mapsto (g2\rho(t)^{-1}, t^2),
\end{equation}
and the left hand side is the more familiar form of Deligne's modification of the Langlands dual group (e.g. see \cite[Proposition 5.3.3]{BG}).
\begin{rmk}\label{R: Cgrp}
We may regard $\hat G^T$ as the dual group of a reductive group $G^T$, which is a central extension of $G$ by $\bG_m$ over $F$, and then 
regard ${^C}G\cong {^L}G^T$ 
as the usual Langlands dual group of $G^T$. This is the definition given in  \cite[Definition 5.3.2]{BG}.
\end{rmk}
We discuss a few examples. 
\begin{ex}
\begin{enumerate}
\item If $G=T$ is a torus, then $\hat{T}^T=\hat{T}\times\bG_m$ and ${^C}T=(\hat{T}\rtimes\Ga_{\tilde F/F})\times\bG_m$.
\item If $G$ is an inner form of a split reductive group, then ${^C}G=\hat{G}^T$. For example, if $G=\on{PGL}_2$, ${^C}G=\hat G^T=\GL_2$ and $d_{\rho_\ad}$ is the usual determinant map. In particular, $\hat{G}^T$ is not isomorphic to $\hat{G}\times\bG_m$ in general.
\item However, if $\rho_\ad$ lifts to a cocharacter $\tilde\rho\in \xcoch(\hat T)$ (which does not necessarily satisfy $2\tilde\rho=2\rho$), then $\tilde\rho$ induces an isomorphism
\begin{equation}\label{E: twisting}
   \hat G^T\simeq \hat G\times \bG_m,\ (g,t)\mapsto (g\tilde\rho(t), t).
\end{equation}
For example, if $G=\GL_n$ so $\hat G=\GL_n$, we can choose $\tilde\rho(t)= \on{diag}\{t^{n-1},\cdots,t,1\}$;
If $G=\on{PGL}_{2n+1}$ so $\hat G=\SL_{2n+1}$, we can choose $\tilde\rho(t)=\on{diag}\{t^{n},t^{n-1},\cdots,t^{-n}\}$.

\item In general, even if $\rho_\ad$ lifts to an element $\tilde\rho\in\xcoch(\hat T)$, so $\hat{G}^T\cong \hat{G}\times\bG_m$, ${^C}G$ may not be isomorphic to ${^L}G\times \bG_m$, unless $\tilde\rho$ can be chosen to be $\Ga_{\widetilde{F}/F}$-invariant.
For example, if $G=\on{U}_{2n}$ is an even unitary group, associated to a quadratic extension $\widetilde{F}/F$, then $\hat{G}=\GL_{2n}$, and $\xi: \Ga_{\widetilde{F}/F}=\{1,c\}\to \Aut(\GL_{2n})$ is defined by
$\xi(c)(A)=J_{2n}(A^T)^{-1}J^{-1}_{2n}$,
where $J_{2n}$ is the anti-diagonal matrix with $(i,2n+1-i)$-entry $(-1)^i$. In this case, 
$\hat{G}^T$ is isomorphic to $\hat{G}\times\bG_m$, but  ${^C}G$ is not isomorphic to ${^L}G\times\bG_m$.
\end{enumerate}
\end{ex}

\subsection{Affine monoids}\label{SS: Aff Mon}

We define an affine monoid scheme $V_{\hat{G},\rho_\ad}$ equipped with a faithfully flat monoid morphism 
\begin{equation}\label{E: monoid ext}
   d_{\rho_\ad}: V_{\hat{G},\rho_\ad}\to \bA^1
\end{equation}    
which extends the homomorphism $d_{\rho_\ad}: \hat{G}^T\to \bG_m$ from \eqref{hat GT}. 
It will be obtained as the pullback of a universal monoid (called the Vinberg monoid) associated to $\hat{G}$, whose definition we first briefly recall following the approach in \cite[\S 2, \S 3]{XZ2}.  Note that the discussions of \emph{loc. cit.} are given over a field. But as they work over any field, they actually work over the base $\bZ$.

We will use notations from \cite{XZ2} with $(\hat{G},\hat{B},\hat{T})$ in place of $(G,B,T)$ in \emph{loc. cit.}. We identify $\xch(\hat{T}_\ad)\subset\xch(\hat{T})$ with the root lattice inside the weight lattice.
Let $\xch(\hat{T}_\ad)_{\on{pos}}$ be the submonoid generated by simple roots $\{\hat{\al}_1,\ldots,\hat{\al}_r\}$ of $\hat{G}$ with respect to $(\hat{B},\hat{T})$, let $\xch(\hat{T})^{?}\subset\xch(\hat{T})$ for $?=+,-$ be the submonoids of dominant and anti-dominant weights, and let $\xch(\hat{T})_{\on{pos}}^+\subset\xch(\hat{T})$ be the submonoid generated by $\xch(\hat{T}_\ad)_{\on{pos}}$ and $\xch(\hat{T})^+$. 
Let 
\[
   \hat{T}_\ad^+=\Spec \bZ[\xch(\hat{T}_\ad)_{\on{pos}}],
\]   
which is an affine monoid containing $\hat{T}_\ad$ as the group of invertible elements. 
Note that every dominant coweight $\la: \bG_m\to \hat{T}_\ad$ of $\hat{G}_\ad$ can be extended to a monoid homomorphism
\begin{equation}\label{E: laplus}
   \la^+:\bA^1\to \hat{T}^+_\ad.
\end{equation}

We equip  $\xch(\hat{T})$ with the usual partial order $\preceq$, i.e. $\hat{\la}_1\preceq \hat{\la}_2$ if $\hat{\la}_2-\hat{\la}_1\in \xch(\hat{T}_\ad)_{\on{pos}}$. For $\hat{\nu}\in \xch(\hat{T})$, let $\hat{\nu}^*:=-w_0(\hat{\nu})$.
Recall that  the left and the right multiplication of $\hat{G}$ on itself induce a natural $(\hat{G}\times\hat{G})$-module structure on $\bZ[\hat{G}]$, which admits a canonical multi-filtration 
\[
   \bZ[\hat{G}]=\sum_{\hat{\nu}\in \xch(\hat{T})_{\on{pos}}^+}\on{fil}_{\hat{\nu}}\bZ[\hat{G}]
\]    
by $(\hat{G}\times\hat{G})$-submodules indexed by $\xch(\hat{T})_{\on{pos}}^+$. 
Here $\on{fil}_{\hat{\nu}}\bZ[\hat{G}]$ denotes the saturated $\bZ$-module that is maximal among all $(\hat{G}\times\hat{G})$-submodules $V\subset \bZ[\hat{G}]$ with the following property: for any pair $(\hat{\mu},\hat{\mu}')\in\xch(\hat T\times \hat T)$, if the weight space $V(\hat{\mu},\hat{\mu}')\neq 0$, then $\hat{\mu}\preceq \hat{\nu}^*$ and $\hat{\mu}'\preceq \hat{\nu}$.
One knows that $\on{fil}_\nu\bZ[\hat{G}]$ is finite free over $\bZ$ and that
\[
   \on{gr} \bZ[\hat{G}]=\bigoplus_{\hat{\nu}\in\xch(\hat{T})^+} S_{\hat{\nu}^*}\otimes S_{\hat{\nu}},
\]
where $S_{\hat{\nu}}$ denotes the Schur module of highest weight $\hat{\nu}$ (i.e. the induced $\hat{G}$-module from the character $-\hat{\nu}$ of $\hat{B}$).
See \cite[\S 2,\S 3]{XZ2} for detailed discussions, including the definition of associated graded of a multi-filtration. Now the Vinberg monoid $V_{\hat{G}}$ of $\hat{G}$ is defined as
\[
    V_{\hat{G}}=\Spec R_{\xch(\hat{T})_{\on{pos}}^+} \bZ[\hat{G}],
\]
where $R_{\xch(\hat{T})_{\on{pos}}^+} \bZ[\hat{G}]:=\oplus_{\hat{\nu}\in\xch(\hat{T})_{\on{pos}}^+} \on{fil}_{\hat{\nu}}\bZ[\hat{G}]$ denotes the Rees algebra associated to the above defined multi-filtration.
It is an affine monoid $\bZ$-scheme of finite type which admits a faithfully flat monoid morphism $d:V_{\hat{G}}\to \hat{T}_\ad^+$ such that
\begin{itemize}
\item $d^{-1}(\hat{T}_\ad)$ contains the group of invertible elements of $V_{\hat{G}}$ and is isomorphic to the quotient of $\hat{G}\times\hat{T}$ by $Z_{\hat{G}}$ with respect to the diagonal action $z\cdot(g,t)=(gz,tz)$ (so in particular $\hat{G}$ acts on $V_{\hat{G}}$ by left and right translations);
\item $d^{-1}(0)\cong \on{As}_{\hat{G}}$ as $(\hat{G}\times \hat{G})$-schemes, where
\[
    \on{As}_{\hat{G}}:= \Spec \on{gr} \bZ[\hat{G}]
\]
is called the asymptotic cone of $\hat{G}$.
\end{itemize}
In fact, the Vinberg monoid of $\hat{G}$ can be characterized as the unique affine monoid scheme $M$ equipped with a faithfully flat monoid morphism $d: M\to \hat{T}_\ad^+$ satisfying the above two properties.

Now, for a dominant coweight $\la:\bG_m\to \hat{T}_\ad$, let 
\begin{equation}\label{e: monoid la}
   d_\la: V_{\hat{G},\la}:=\bA^1\times_{\la^+,\hat{T}_\ad}V_G\to \bA^1
\end{equation}
be the pullback of of the homomorphism $d:V_{\hat{G}}\to \hat{T}^+_\ad$ along the map $\la^+$ from \eqref{E: laplus}. 
The homomorphism $\hat{G}\times\bG_m\xrightarrow{\id\times 2\rho}\hat{G}\times\hat{T}$ induces a homomorphism $\hat{G}^T\to \hat{G}\times^{Z(\hat{G})}\hat{T}$ by \eqref{R: another version}, and therefore
we obtain desired map \eqref{E: monoid ext} by setting $\la=\rho_\ad$ in \eqref{e: monoid la}.

Note that $V_{\hat{G}}$ and $\hat{T}^+_\ad$ are acted by $\Ga_{\widetilde{F}/F}$ (induced by its action $\xi$) and the homomorphism $d: V_{\hat{G}}\to \hat{T}^+_\ad$ is $\Ga_{\widetilde{F}/F}$-equivariant. If $\la:\bG_m\to \hat{T}_\ad$ is a dominant coweight fixed by $\Ga_{\widetilde{F}/F}$, the monoid $V_{\hat{G},\la}$ is also acted by $\Ga_{\widetilde{F}/F}$ and the map $d_{\la}$ extends to a homomorphism 
\[
   \tilde{d}_{\la}: V_{\hat{G},\la}\rtimes \Ga_{\widetilde{F}/F} \to\bA^1\times \Ga_{\widetilde{F}/F}.
\]
Note that $\rho_\ad$ is $\Ga_{\widetilde{F}/F}$-invariant. It follows that there is a natural isomorphism 
\begin{equation}\label{E: group unit in Vrho}
   \tilde{d}_{\rho_\ad}^{-1}(\bG_m\times\Ga_{\widetilde{F}/F})\cong {^C}G.
\end{equation}

\subsection{Invariant theory}
We fix a finite order automorphism $\sigma$ of $(\hat{G},\hat{B},\hat{T},\hat e)$ and a positive integer $q$. We consider the $\sigma$-twisted conjugation action of $\hat{G}$ on $V_{\hat{G}}$ given by
\begin{equation*}
   c_\sigma(g)(x)=gx\sigma(g)^{-1},\quad g\in \hat{G},\ x\in V_{\hat{G}}.
\end{equation*}
Let $\la: \bG_m\to \hat T_\ad$ be a dominant coweight. Then the conjugation action of $\hat{G}$ on $V_{\hat G}\rtimes\langle\sigma\rangle$ restricts to an action on $V_{\hat G}|_{d=\la(q)}\sigma$. 
Then we have the isomorphism 
$\bZ[V_{\hat G}|_{d=\la(q)}\sigma]^{\hat G}=\bZ[V_{\hat{G}}|_{d=\la(q)}]^{c_\sigma(\hat G)}$.

Let $V_{\hat{T}}$ be the closure of $\hat{T}\times^{Z_{\hat{G}}} \hat{T}\subset \hat{G}\times^{Z_{\hat{G}}}\hat{T}$ inside $V_{\hat{G}}$. This is a commutative submonoid of $V_{\hat{G}}$.
If we write the ring of regular functions of $\hat{T}\times^{Z_{\hat{G}}} \hat{T}$ by 
$$
   \bZ[\hat{T}\times^{Z_{\hat{G}}}\hat{T}]=\bigoplus_{(\hat{\la},\hat{\nu})}\bZ (e_1^{\hat{\la}}\otimes e_2^{\hat{\nu}}), 
$$
where the sum is taken over  pairs of weights $(\hat{\la},\hat{\nu})$ of $\hat{T}$ such that $\hat{\nu}+\hat{\la}\in\xch(\hat{T}_\ad)$\footnote{In \cite{XZ2} the condition reads as $\hat{\nu}-\hat{\la}\in \xch(\hat{T}_\ad)$. The sign convention here is more suitable for our purpose.}, and where $e_1^{\hat{\la}}\otimes e_2^{\hat{\nu}}$ denotes the corresponding regular function on $\hat{T}\times^{Z_{\hat{G}}}\hat{T}$,
then the ring of regular functions on $V_{\hat{T}}$ is the subring
\[
\bZ[V_{\hat{T}}]=\bigoplus_{(\hat{\la},\hat{\nu}), \hat{\nu}+\hat{\la}_{-}\in \xch(\hat{T}_\ad)_{\on{pos}}}\bZ (e_1^{\hat{\la}}\otimes e_2^{\hat{\nu}}),
\]
where
$\hat{\la}_{-}\in\xch(\hat T)^-$ denotes the anti-dominant weight in the Weyl group orbit of $\hat{\la}$. 

For $\hat\la\in \xch(\hat T)$, let $e^{\hat\la}$ denote the corresponding regular function on $\hat T$, and for $\hat\la\in \xch(\hat{T}_\ad)_{\on{pos}}$, let $\bar{e}^{\hat\la}$ be the corresponding regular function on $\hat{T}_\ad^+$.
The homomorphism $d: V_{\hat{T}}\to \hat{T}_\ad^+$ is given by the ring map $\bZ[\xch(\hat{T}_\ad)_{\on{pos}}]\to \bZ[V_{\hat{T}}]$ sending $\bar{e}^{\hat\la}$ to $1\otimes e_2^{\hat\la}$, which
admits a section $\fraks: \hat{T}_\ad^+\to V_{\hat{T}}$
\begin{equation*}
     \bZ[V_{\hat{T}}]\to \bZ[\xch(\hat{T}_\ad)_{\on{pos}}], \quad e^{\hat\la}_1\otimes e^{\hat\nu}_2\mapsto \bar{e}^{(\hat\nu+\hat\la)}.
\end{equation*}
Note that $\fraks|_{\hat{T}_\ad}: \hat{T}_\ad\to \hat{T}\times^{Z_{\hat{G}}}\hat{T}$ is induced by the diagonal embedding $\hat{T}\to \hat{T}\times\hat{T}$. We still denote by $\fraks$ the composed section $\hat{T}_\ad^+\to V_{\hat{G}}$. Its restriction to $\hat{T}_\ad$ induces a section $\hat{T}_\ad\to \hat{G}\times^{Z_{\hat{G}}}\hat{T}$, and therefore an isomorphism
$\hat{G}\times^{Z_{\hat{G}}}\hat{T}\cong \hat{G}\rtimes\hat{T}_\ad$, whose pullback along $\rho_\ad:\bG_m\to \hat{T}_\ad$ gives the isomorphism \eqref{R: another version}.

There is a natural injective map $i_1:\hat{T}\to \hat{T}\times^{Z_{\hat{G}}}\hat T\to V_{\hat{T}}$, where the first map is the inclusion into the first factor. Then we obtain a map
$(i_1,\fraks): \hat{T}\times \hat{T}_\ad^+\to V_{\hat T}$ over $\hat{T}_\ad^+$, which induces the ring map 
$$
   \bZ[V_{\hat {T}}]\to \bZ[\xch(\hat T)]\otimes\bZ[\xch(\hat T_\ad)_{\on{pos}}],\quad e_1^{\hat\la}\otimes e_2^{\hat\nu}\mapsto e^{\hat\la}\otimes \bar{e}^{(\hat\nu+\hat\la)}.
$$ 
It in turn induces an injective map
\begin{equation}\label{eq: inj function on VT}
\bZ[V_{\hat{T}}|_{d=\la(q)}]\subset \bZ[\xch(\hat{T})],    \quad e_1^{\hat\la}\otimes e_2^{\hat\nu}\mapsto q^{\langle\la,\hat\nu + \hat\la\rangle}e^{\hat\la}.
\end{equation}

Let $W=N_{\hat{G}}(\hat{T})/\hat{T}$ be the Weyl group of $(\hat{G},\hat{T})$ and let $W_0=W^\sigma$ be the subgroup of elements fixed by $\sigma$, which naturally acts on $\xch(\hat T)^\sigma$. Let $\hat{N}_0$ be the preimage of $W_0$ in $N_{\hat{G}}(\hat{T})$. The $\sigma$-twisted conjugation $c_\sigma$ of $\hat{G}$ on $V_{\hat G}$
induces the $\sigma$-twisted action of $\hat{N}_0$ on $V_{\hat{T}}$.
Recall that there is the Chevalley restriction isomorphism (see \cite[Proposition 4.2.3]{XZ2}, which was denoted as $\Res^\sigma_{+,\mathbf{1}}$)
\[
   \Res: \bZ[V_{\hat{G}}]^{c_\sigma(\hat{G})}\cong \bZ[V_{\hat T}]^{c_\sigma(\hat{N}_0)},
\] 
compatible with the $\bZ[\xch(\hat{T}_\ad)_{\on{pos}}]$-structure on both sides. The same argument as in \cite[Lemma 4.2.1]{XZ2} implies that the isomorphism $\Res$
induces isomorphisms
\begin{multline}\label{E: Chev res q}
   \Res: \bZ[V_{\hat{G}}|_{d=\la(q)}]^{c_\sigma(\hat{G})}\cong \bZ[V_{\hat{G}}]^{c_\sigma(\hat{G})}\otimes_{\bZ[\xch(\hat T_\ad)_{\on{pos}}],\bar{e}^{\hat{\al}_i}\mapsto q^{(\la,\hat{\al}_i)}}\bZ \\
   \cong \bZ[V_{\hat T}]^{c_\sigma(\hat{N}_0)}\otimes_{\bZ[\xch(\hat T_\ad)_{\on{pos}}],\bar{e}^{\hat{\al}_i}\mapsto q^{(\la,\hat{\al}_i)}}\bZ \cong \bZ[V_{\hat T}|_{d=\la(q)}]^{c_\sigma(\hat{N}_0)}.
\end{multline}

\begin{rmk}
The $c_\sigma$-action of $\hat{N}_0$ on $V_{\hat T}|_{d=\la(q)}$ induces an action of $\hat{N}_0$ on $\bZ[V_{\hat{T}}|_{d=\la(q)}]$. On the other hand, the $c_\sigma$-action of $\hat{N}_0$ on $\hat{T}$ induces an action of $\hat{N}_0$ on $\bZ[\xch(\hat{T})]$. The inclusion \eqref{eq: inj function on VT} is \emph{not} equivariant with respect to these two actions. Indeed, the base change of \eqref{eq: inj function on VT} to $\bQ$ becomes an isomorphism, under which the action of $\hat{N}_0$ on $\bQ[V_{\hat{T}}|_{d=\la(q)}]$ induces an action of $W_0$ on $\bQ[\xch(\hat T)^\sigma]$
given by
\begin{equation}\label{E: twist W-action}
w\bullet_{\la} e^{\hat\la}= q^{\langle \la, w\hat\la-\hat\la\rangle}  e^{w\hat\la},\quad w\in W_0,\  \hat\la\in\xch(\hat T)^\sigma.
\end{equation}
\end{rmk}

The following lemma follows from \eqref{eq: inj function on VT}.
\begin{lem}\label{L: image Chevalley restriction}
The image of the map 
\begin{equation}\label{eq: inj inv function on VT}
  \bZ[V_{\hat T}|_{d=\la(q)}]^{c_\sigma(\hat{N}_0)}\subset \bZ[V_{\hat T}|_{d=\la(q)}]\xrightarrow{\eqref{eq: inj function on VT}} \bZ[\xch(\hat T)]
\end{equation}
is the subring of $\bZ[\xch(\hat T)^\sigma]$ with a $\bZ$-basis given by elements of the form 
\[
   \sum_{\hat\la'\in W_0\hat\la}   q^{\langle\la,\hat\la'-\hat\la\rangle} e^{\hat\la'},\quad \hat \la\in \xch(\hat T)^{\sigma,-}:= \xch(\hat T)^\sigma\cap \xch(\hat T)^-.
\]   
\end{lem}

\subsection{Classical Satake isomorphism over $\bZ$}\label{SS: statement}
Now, we assume that $F$ is a non-archimedean local field, with $\mO$ its ring of integers and $k\simeq \bF_q$ the residue field. Let $G$ be a connected reductive group over $\mO$. 
In this case, $\widetilde{F}/F$ is an unramified extension and $\Ga_{\widetilde{F}/F}\cong\langle\sigma\rangle$ is generated by the \emph{geometric} $q$-Frobenius $\sigma$ of $k$\footnote{Using the arithmetic Frobenius will lead to a different formulation by taking the dual.}. We also fix a uniformizer $\varpi\in \mO$. Then every $\hat\la\in\xcoch(T)^\sigma$ gives a well defined point $\hat\la(\varpi)\in T(F)\subset G(F)$.

Let $K=G(\mO)$, and let $H_G=C_c(K\backslash G(F)/ K,\bZ)$ be the space of compactly supported bi-$K$-invariant $\bZ$-valued functions on $G(F)$. This is a natural algebra under convolution (with the chosen Haar measure on $G(F)$ such that the volume of $K$ is $1$). Let $T$ be the abstract Cartan of $G$, which is defined as the quotient of a Borel subgroup $B\subset G$ over $\mO$ by its unipotent radical $U\subset B$. (But $T$ is canonically defined, independent of the choice of $B$, e.g. see \cite[0.3.2]{Z17b}).
For a choice of Borel subgroup $B\subset G$ over $\mO$, we define the classical Satake transform
\begin{equation}\label{E: Sat tran}
    \on{CT}^{\cl}: H_G\to \bZ[\xch(\hat{T})^\sigma],\quad f\mapsto  \on{CT}^{\cl}(f)=\sum_{\hat{\la}\in\xch(\hat{T})^\sigma} \bigl(\sum_{u\in U(F)/U(\mO)}f(\hat{\la}(\varpi) u)\bigr)e^{\hat{\la}}.
\end{equation}
This map is independent of the choice of $B$ and the uniformizer $\varpi$.

\begin{prop}\label{prop: main}
There exists a unique isomorphism 
$$\Sat^{\cl}: \bZ[V_{\hat{G},\rho_{\ad}}\rtimes\langle\sigma\rangle|_{\tilde d_{\rho_\ad}=(q,\sigma)}]^{\hat{G}}\xrightarrow{\cong}  H_G,$$ 
which we call the Satake isomorphism, making the following 
diagram commutative
\[
\xymatrix{
  \bZ[V_{\hat{G},\rho_{\ad}}\rtimes\langle\sigma\rangle|_{\tilde d_{\rho_\ad}=(q,\sigma)}]^{\hat{G}}\ar_-\cong^-{\Sat^{\cl}}[rr] \ar_{\Res}^-\cong[d]&& H_G \ar^-{\on{CT}^{\cl}}[d]\\
\bZ[V_{\hat{T}}|_{d=\rho_\ad(q)}]^{c_\sigma(\hat{N}_0)}\ar@{^{(}->}^-{\eqref{eq: inj inv function on VT}}[rr] && \bZ[\xch(\hat{T})^\sigma].
}
\]
In particular, $H_G$ is finitely generated.
\end{prop}
\begin{proof}
The uniqueness is clear.  To prove the existence, by Lemma \ref{L: image Chevalley restriction}, it is enough to show that the Satake transform \eqref{E: Sat tran} induces an isomorphism
\[
     \on{CT}^{\cl}: H_G\xrightarrow{\cong} \bZ[\xch(\hat{T})^\sigma]\cap \bQ[\xch(\hat{T}^\sigma]^{W_0\bullet_{\rho_\ad}},
\]
where $W_0\bullet_{\rho_\ad}$ denotes the action given in \eqref{E: twist W-action} (with $\la=\rho_\ad$). Indeed, this follows from the usual Satake isomorphism by noticing that $\eqref{E: Sat tran}$ differs from the usual Satake transform (e.g. see \cite[(3.4)]{Gr} in the split case) by a square root of the modular character.
\end{proof}

\begin{rmk}
The above isomorphism might look artificial as we identify both sides with a subring of $\bZ[\xch(\hat T)^\sigma]$. In the next section, we will deduce this isomorphism from the geometric Satake\footnote{In fact, this is how the formulation given here was first discovered.}. This alternative approach has the advantage that it is more natural and is independent of the usual Satake isomorphism, and is useful for some arithmetic applications.
\end{rmk}

Let us explain the relation of $\Sat^\cl$ with the classical Satake isomorphism (e.g see \cite[Proposition 3.6]{Gr} in the split case) and the mod  $p$ Satake isomorphism as in \cite{He,HV}. 

First by \eqref{E: group unit in Vrho}, $(V_{\hat{G},\rho_\ad}|_{\tilde d_{\rho_\ad}=(q,\sigma)})_{\bZ[q^{-1}]}\cong ({^C}G|_{d_{\rho_\ad}=(q,\sigma)})_{\bZ[q^{-1}]}$, so the above isomorphism induces a canonical isomorphism
\[
   H_G\otimes\bZ[q^{-1}]\cong \bZ[q^{-1}][{^C}G|_{d_{\rho_\ad}=(q,\sigma)}]^{\hat{G}}.
\]
After choosing a square root $q^{1/2}$, there is a $\hat{G}$-equivariant isomorphism (comparing with \eqref{E: twisting})
\[
 {^C}G|_{d_{\rho_\ad}=(q,\sigma)}\cong \hat{G}\sigma,\quad (g,(q,\sigma))\in \hat{G}\rtimes(\bG_m\times\langle\sigma\rangle)\mapsto g2\rho(q^{-\frac{1}{2}})\sigma \in {^L}G. 
\]
Note that the following diagram is commutative
\[\xymatrix{
 \bZ[q^{\pm\frac{1}{2}}][{^C}G|_{d_{\rho_\ad}=(q,\sigma)}]^{\hat G}\ar^-\cong[rr]\ar_-{\eqref{eq: inj inv function on VT}\circ\Res}[d] && \bZ[q^{\pm\frac{1}{2}}][\hat G\sigma]^{\hat G}\ar^{\Res}[d] \\
  \bZ[q^{\pm\frac{1}{2}}][\xch(\hat T)^\sigma] \ar^{e^{\hat \la}\mapsto (q^{-\frac{1}{2}})^{(2\rho,\hat\la)}e^{\hat\la}}[rr]&&  \bZ[q^{\pm\frac{1}{2}}][\xch(\hat T)^\sigma],
}\]
where the right vertical map is the restriction map from functions on $\hat{G}\sigma$ to function on $\hat{T}\sigma$.
The composition of \eqref{E: Sat tran} with the bottom map in the above diagram is the usual Satake transform.
We thus obtain the usual classical Satake isomorphism 
\[
  H_G\otimes\bZ[q^{\pm \frac{1}{2}}]\cong \bZ[q^{\pm \frac{1}{2}}][\hat{G}\sigma]^{\hat{G}}.
\]

On the other hand, after mod $p$,
\eqref{E: Sat tran} is the formula used in \cite{He, HV} to define the mod $p$ Satake isomorphism. 
In addition, $\bF_p[V_{\hat{G}}|_{d=\rho_\ad(q)}]=\bF_p[\on{As}_{\hat{G}}]$.
\begin{cor}
There is a canonical isomorphism $H_G\otimes\bF_p\cong \bF_p[\on{As}_{\hat{G}}]^{c_\sigma(\hat{G})}$.
\end{cor}
This gives a natural description of the mod $p$ Hecke algebra by Langlands duality. Note that by definition
\[
   \bF_p[\on{As}_{\hat{G}}]^{c_\sigma(\hat{G})}= \bigoplus_{\nu\in\xch(\hat{T})^+} (S_{\nu^*}\otimes S_{\nu})^{c_\sigma(\hat{G})}\cong \bF_p[\xcoch(T)^{\sigma,-}].
\]
Therefore, we recover the mod $p$ Satake isomorphism \cite{He, HV} (for trivial $V$ in \emph{loc. cit.}).

\begin{ex}
Let $G=\on{PGL}_2$, so that ${^C}G=\hat{G}^T=\GL_2$ and $V_{\hat{G},\rho_\ad}=V_{\hat{G}}=M_2$ is the monoid of $2\times 2$-matrices, and $d=\det: M_2\to \bA^1$ is the usual determinant map. Then $\bZ[M_2|_{\det=q}]^{\SL_2}\cong \bZ[\tr]$ is the polynomial ring generated by the trace function. On the other hand, $H_G=\bZ[T_p]$ is a polynomial ring generated by the $T_p$-operator. Under the canonical Satake isomorphism $T_p$ matches $\tr$.
\end{ex}

\begin{rmk}

(1) Proposition \ref{prop: main} is compatible with the Weil restriction of $G$ along unramified extensions. We leave the verification as an exercise.

(2) As suggested by Bernstein,  it is the $C$-group rather than the $L$-group that should be used in the formulation of the Langlands functoriality. Similarly, we expect the Vinberg monoid of $\hat{G}$ might be useful to formulate the more subtle arithmetic aspect of the Langlands functoriality.
\end{rmk}


\subsection{Satake isomorphism for the Hecke algebra of weight $V$}\label{S: padic integral}
We retain notations from \S \ref{SS: statement}.
Let $\La$ be a $\bZ$-algebra, and let $(V,\pi)$ be a $\La$-module equipped with a $\La$-linear action of $K=G(\mO)$.

We first briefly recall the general formalism of Hecke algebra $H_G(V)$ of ``weight" $V$ and the Satake transform. We refer to \cite{HV} for a more general treatment in a more abstract setting. First, we define
\begin{equation}\label{E: Hk alg wt}
  H_G(V):= \{f: G(F)\to \End_{\La}(V)\mid f(k'gk)(v)=k'(f(g)(kv)), \forall k,k'\in K,\ \ \ \on{Supp}(f) \mbox{ is compact}\},
\end{equation}
with the ring structure given by convolution
\[
   (f_1\star f_2)(g)(v)=\sum_{h\in G/K}f_1(h)(f_2(h^{-1}g) (v)),\quad g\in G(F), v\in V.
\]
\begin{ex}
Let $G=T$ be a torus over $\mO$. Let $\La=\mO_L$, where $L$ is a non-archimedean local field over $F$. Let $V$ be the rank one free module over $\La$ on which $T(\mO)$ acts through a continuous character $\chi: T(\mO)\to \mO_L^\times$. In this case, we write $H_T(V)$ as $H_T(\chi)$. There is an isomorphism
\begin{equation}\label{E: Sat wt tori}
   H_T(\chi)\cong \mO_L[\xch(\hat T)^\sigma],  \quad f\mapsto \sum_{\hat\la\in\xch(T)^\sigma}   f(\hat\la(\varpi))e^{\hat\la}.
\end{equation}
If $\chi$ is non-trivial, this isomorphism depends on the choice of a uniformizer $\varpi\in F$. 
\end{ex}

Similar to \eqref{E: Sat tran}, there is the Satake transform
\begin{equation}\label{E: Sat tran weight}
  \on{CT}_V^{\cl}: H_G(V)\to  H_T(V^{U(\mO)}), \quad f\mapsto \Big(\on{CT}_V^{\cl}(f): t\in T(F)\mapsto \sum_{u\in U(F)/U(\mO)} f(tu)|_{V^{U(\mO)}}\Big).
\end{equation}
To justify the definition, note that the sum $\sum_{u\in U(F)/U(\mO)} f(tu)|_{V^{U(\mO)}}$ is finite and that
 for $v\in V^{U(\mO)}$, 
$\sum_{u\in U(F)/U(\mO)} f(tu)(v)\in V^{U(\mO)}$. In addition, one verifies that \eqref{E: Sat tran weight} is a homomorphism, either
by a direct computation (e.g. see Step 2 in the proof of \cite[Theorem 1.2]{He}), or by considering the action of $H_G(V)$ on the ``principal series representation of weight $V$" (e.g. see \cite[Section 2]{HV}).
By virtue of \eqref{E: Sat wt tori},
 \eqref{E: Sat tran weight} specializes to \eqref{E: Sat tran} when $V=\mathbf{1}$ is the trivial representation.
\begin{rmk}
In most literature, $V$ is assumed to be a smooth $K$-module, i.e. the stabilizer of every element $v\in V$ in $K$ is open. But this assumption is in fact not necessary in the above discussions.
\end{rmk}
\quash{with the usual $G(F)$ action given by $(g\cdot f)(g')=f(g^{-1}g')$. For $v\in V$, let $f_v\in \ind_K^{G(F)}V$ be the function supported on $K$ given by $f_v(k)=k^{-1}v$.
There is a natural ring isomorphism
\[
\End_{G(F)}(\ind_K^{G(F)}V)\cong H(G,K,V)
\]
sending $\phi: \ind_K^{G(F)}V\to\ind_K^{G(F)}V$ to $f_\phi: G(F)\to \End_{\mO}(V)$, where
\[
    f_\phi(g)(v)= (g\phi(f_v))(1).
\]
Under this bijection, compositions in $\End_{G(F)}(\ind_K^{G(F)}V)$ become convolutions of $\{f: G(F)\to \End_{\mO}(V)\mid f(k'gk)(v)=k'(f(g)(k^{-1}v))\}$.
}

In the above generality, there is very little one can say about \eqref{E: Sat tran weight}.
In the sequel, 
we specialize $V$ to the following situation. 
Let $L$ be a non-archimedean local field over $F$, with $\mO_L$ its ring of integers.
Let
$V$ be a finite free $\mO_L$-module arising from an algebraic representation of $G$ over $\mO_L$, such that $V^{U(\mO)}$ is free of rank one. 
This is the case if and only if $\dim_LV_L^{U_L}=1$. In this case, $T(\mO)=B(\mO)/U(\mO)$ acts on $V^{U(\mO)}$ via a dominant weight $\la$ of $T$. It follows from \eqref{E: Sat wt tori} that $H_T(V^{U(\mO)})=H_T(\la)$ is commutative. 
\begin{lem}\label{L: inj Sat} 
Under the above assumption, the map $\on{CT}^{\cl}_{V}$ is injective.
\end{lem}
\begin{proof}
The proof given below follows the same strategy in \cite{ST,He,HV}. First,
\quash{
\begin{ex}
Let $G=\GL_n$ over $\mO$.
\end{ex}

To see that this is a ring homomorphism, we use the usual strategy to consider the action of $H_G(V)$ on the space
\[
    C( (K,V)\backslash G(F)/U(F)):=\{\xi: G(F)/U(F)\to V \mid \xi(kg)=k \xi(g), \}
\]
by convolution
\[
    (f\star\xi)(g)= \sum_{h\in G/K}f(h)(\xi(h^{-1}g)),\quad  g\in G(F)/U(F).
\]
Note that $\xi(\hat\la(\varpi))\in V^{\hat\la(\varpi) U(\mO) \hat\la(\varpi)^{-1}\cap U(\mO)} = V^{U(\mO)}$. Then
$\xi$ is uniquely determined by $\{\xi(\hat\la(\varpi))\}_{\hat\la\in\xch(\hat T)^\sigma}$. For each $\la$, we have a unique $\xi_\la$, such that $\xi_\la(\hat\la(\varpi))=\varpi^{\langle\la,-\hat\la\rangle}v$. Then the action of $H_{G}(V)$ preserves each $\xi_\la$, and
\[
   (f\star\xi_\la)(1)=\la(\on{CT}^{\cl}(f)) v,
\]
where $\la: \mO[\xch(\hat T)^\sigma]\to \mO$ sends $e^{\hat\la}$ to $\varpi^{\langle\la,\hat\la\rangle}$.

It follows that the Satake transform \eqref{E: Sat tran weight} is an algebra homomorphism.

}

\begin{lem}\label{L: base wt Hecke}
Fix a uniformizer $\varpi\in F$. For every $\hat\mu\in\xch(\hat T)^{\sigma,-}$,
There is a unique element $f_{\hat\mu}\in H_G(V)$ satisfying 
\begin{itemize}
\item $f_{\hat\mu}$ is supported on $K\hat\mu(\varpi)K$;
\item $f_{\hat\mu}(\hat\mu(\varpi))|_{V^{U(\mO)}}=\id: V^{U(\mO)}\to V^{U(\mO)}$.
\end{itemize}
In addition,
The collection $\{f_{\hat\mu}\}_{\hat\mu\in\xch(\hat T)^{\sigma,-}}$ form an $\mO_L$-basis of $H_G(V)$.
\end{lem}
\begin{proof}
Clearly, restricting a map $f: G(F)\to \End_{\mO_L}(V)$ to $\hat\mu(\varpi)$ induces a bijection between elements in $H_{G}(V)$ satisfying the first condition and $\mO_L$-linear maps $\varphi:V\to V$ satisfying 
\begin{equation}\label{E: equiv stable}
   \pi(\hat{\mu}(\varpi)k\hat{\mu}(\varpi)^{-1}) \varphi(w)= \varphi(\pi(k)w), \quad \forall k\in K\cap \hat{\mu}(\varpi)^{-1} K \hat{\mu}(\varpi), w\in V.
\end{equation}
As $K\cap \hat{\mu}(\varpi)^{-1} K \hat{\mu}(\varpi)$ is Zariski dense in $G(F)$, the rational map $\varphi_L: V_L\to V_L$ must be equal to $c\pi(\hat{\mu}(\varpi))$ for some $c\in L$. Then $\varphi$ preserves the integral lattice $V$ if and only if
$c\in \varpi^{\langle\la,-\hat\mu\rangle}\mO_L$. 

It follows from the above considerations that
$f_{\hat\mu}(\hat\mu(\varpi))= \varpi^{\langle\la,\hat\mu\rangle}\pi(\hat\mu(\varpi))$ is the desired element as in the lemma. In addition 
$\{f_{\hat\mu}\}_{\hat\mu\in\xch(\hat T)^{\sigma,-}}$ form an $\mO_L$-basis of $H_G(K)$.
\end{proof}

Note that the natural map $H_G(V)\otimes L\to H_G(V_L)$ is an isomorphism, and 
there is the following commutative diagram
\begin{equation}\label{E: triv to gen wt}
\xymatrix{
   H_G(\mathbf{1})\otimes L \ar^{\on{CT}^{\cl}_{\mathbf{1}}}[r]\ar[d]&  H_T(\mathbf{1})\otimes L \ar[d]\\
   H_G(V)\otimes L\ar^-{\on{CT}^{\cl}_{V}}[r]&  H_T(V^{U(\mO)})\otimes L
}\end{equation}
where the left vertical map sends $f: G(F)\to L$ to $\tilde{f}: G(F)\to \End(V_L), \ g\mapsto f(g)\pi(g)$, and the right vertical map sends $f: T(F)\to L$ to $\tilde{f}: T(F)\to \End(V_L^{U_L})=L,\ t\mapsto f(t)\la(t)$. 

The left vertical map sends $1_{K\hat\mu(\varpi)K}$ to $\varpi^{\langle\la, \hat\mu\rangle}f_{\hat\mu}$, and therefore is an isomorphism. Similarly, the right vertical map is an isomorphism. As $\on{CT}^{\cl}_{\mathbf{1}}\otimes L$ is injective, we conclude Lemma \ref{L: inj Sat} .
\end{proof}

In this sequel, we further assume that $F=\bQ_p$, and choose the uniformizer $\varpi=p$. Then we write the Satake transform \eqref{E: Sat tran weight} as a homomorphism $\on{CT}^{\cl}_V: H_G(V)\to \mO_L[\xch(\hat T)^\sigma]$
using \eqref{E: Sat wt tori}.

For a dominant weight $\la$ of $G$, let $\la_\ad:\bG_m\to \hat{T}\to\hat{T}_\ad$ be the corresponding cocharacter of $\hat{T}_\ad$.
\begin{prop}\label{prop: main weights}
There exists a unique isomorphism 
\[
  \Sat^{\cl}: \mO_L[V_{\hat{G}}|_{d=(\la_{\ad}+\rho_{\ad})(p)}]^{c_\sigma(\hat{G})}\xrightarrow{\cong}  H_G(
  V), 
\]  
making the following diagram commutative
\[
\xymatrix{
\mO_L[V_{\hat{G}}|_{d=(\la_\ad+\rho_\ad)(p)}]^{c_\sigma(\hat{G})} \ar_{\Res}^-\cong[d]\ar_-\cong^-{\Sat^{\cl}}[rr]&& H_G(V) \ar^-{\on{CT}^{\cl}_V}[d]\\
 \mO_L[V_{\hat{T}}|_{d=(\la_\ad+\rho_\ad)(p)}]^{c_\sigma(\hat{N}_0)}\ar@{^{(}->}^-{\eqref{eq: inj inv function on VT}}[rr] && \mO_L[\xch(\hat{T})^\sigma].
}
\]
In particular, $H_G(V)$ is finitely generated.
\end{prop}
\begin{proof}As in the proof of Proposition \ref{prop: main}, it is enough to prove
\[\on{CT}_V^{\cl}: H_G(V)\cong \mO_L[\xch(\hat T)^\sigma]\cap L[\xch(\hat T)^\sigma]^{W_0\bullet_{\la_\ad+\rho_\ad}}.\]
But this follows from the case $V=\mathbf{1}$ and the commutative diagram \eqref{E: triv to gen wt}.
\end{proof}
\quash{
On the other hand, the right vertical map intertwines the $W_0$-action $\bullet_{\rho}$ and $\bullet_{\la+\rho}$. It follows that the the image of $\on{CT}^{\cl}_V$ is invariant under $\bullet_{\la+\rho}$-action by $W_0$.

\begin{rmk}
The formulation of Hecke algebra \eqref{E: Hk alg wt} does not commute with base change in general. Namely, if $\La\to \La'$ is a ring homomorphism, and $V$ is a $\La[K]$-module, the natural map $H_G(V)\otimes \La'\to H_G(V\otimes\La')$ may not be an isomorphism. But in our situation, clearly $H_G(V)\otimes L\to H_G(V\otimes L)$ is an isomorphism. In addition, when the highest weight $\la$ is ``small" and $V$ is the corresponding Schur module, then $H_G(V)\otimes \kappa_L\to H_G(V\otimes\kappa_L)$ is also an isomorphism, where $\kappa_L$ is the residue field of $\mO_L$,as explained in \cite[Prop. 2.10]{He}. 
\end{rmk}
}

\begin{rmk}
(1) It follows from \eqref{E: Chev res q} that the algebra $\mO_L[V_{\hat{G}}]^{c_\sigma(\hat{G})}$ specializes all $H_G(V)$'s.

(2) Taking the $p$-adic completion and inverting $p$ allows us to recover some results of \cite{ST} on the Banach-Hecke algebra $\widehat{H}_G(V)[1/p]$.

(3) Again, the way to identify $H_G(V)$ with $\mO[V_{\hat{G}}|_{d=(\la_\ad+\rho_\ad)(p)}]^{c_\sigma(\hat{G})}$ given above might look artificial. It would be interesting to have a geometric version of Proposition \ref{prop: main weights}.
\end{rmk}

\section{Compatibility with the geometric Satake}
In this section, we deduce Proposition \ref{prop: main} from the geometric Satake equivalence.
\subsection{The geometric Satake equivalence}\label{SS: geom Sat}
We refer to \cite{Z17a,Z17b} and references cited there for detailed discussions of the geometric Satake equivalence. We retain notations from \S \ref{SS: statement}.
Let $\tilde{F}/F$ be the splitting field of $G$. It is a finite unramified extension of $F$, with $\tilde{\mO}$ its ring of integers and $\tilde{k}$ the residue field. Let $r=[\widetilde{F}:F]=[\tilde{k}:k]$.
Let $L^+G$ denote the positive loop group of $G$ over $k$ and let $\Gr$ denote its affine Grassmannian over $k$.

For $\hat{\mu}\in\xch(\hat{T})^+$,
let $k_{\hat{\mu}}\subset\tilde{k}$ be its field of definition and $d_{\hat{\mu}}:=[k_{\hat{\mu}}:k]$. Let $\Gr_{\leq \hat{\mu}}$ denote the Schubert variety corresponding to $\hat{\mu}$, which is a (perfect) projective scheme defined over $k_{\hat \mu}$. Let $\Gr_{\hat \mu}$ denote the open Schubert cell.

We fix $\ell\neq p$. Let $\IC_{\hat \mu}$ be the intersection complex with $\bQ_\ell$-coefficient on $\Gr_{\leq \hat \mu}$, so that 
\[
   \IC_{\hat \mu}|_{\Gr_{\hat \mu}}=\bQ_\ell[\langle2\rho,\hat\mu\rangle].
\]     
If $k'/k$ is an algebraic extension in $\bar{k}$, let
$\Sat_{G,k',\ell}$ denote the category of $L^+G\otimes k'$-equivariant perverse sheaves on $\Gr\otimes k'$ with $\bQ_\ell$-coefficients, which is a tensor abelian categories. 
Inside $\Sat_{G,\tilde{k},\ell}$, there is a full semisimple monoidal abelian subcategory $\Sat_{G,\tilde{k},\ell}^T$
as defined in \cite{Z17b}\footnote{Strictly speaking, only equal characteristic version was considered in \cite{Z17b}. However its counterpart in mixed characteristic is obvious, using \cite{Z17a}. The similar remark applies to the discussions in the sequel.}: it is the full semisimple tensor abelian category generated by all
$\{\IC_{\hat\mu}(i), \hat\mu\in \xcoch(T)^+, i\in\bZ\}$. Now we define $\Sat_{G,\ell}^T$ as the category of $\Gal(\tilde{k}/k)$-equivariant objects in $\Sat_{G,\tilde{k},\ell}^T$. 
I.e., objects are pairs $(\mF,\gamma)$, where $\mF\in \Sat_{G,\tilde{k},\ell}$ and $\gamma: \sigma^*\mF\simeq \mF$ is an isomorphism such that the induced isomorphism $\mF=(\sigma^r)^*\mF\xrightarrow{\ga\sigma(\ga)\cdots\sigma^{r-1}(\ga)} \mF$ is the identity map, and morphisms from $(\mF,\ga)$ to $(\mF',\ga')$ are morphisms from $\mF$ to $\mF'$ in $\Sat_{G,\tilde{k},\ell}^T$ that are compatible with $\ga$ and $\ga'$.
This is still a semisimple tensor category. 

For a (not necessarily connected) split reductive group $H$ over a field $E$ of characteristic zero, let  $\on{Rep}(H_E)$ denote the category of finite dimensional algebraic representations of $H$ over $E$. Let $\sigma\Mod_{\bQ_\ell}$ denote the category of representations of $\sigma$ on finite dimensional $\bQ_\ell$-vector spaces.
Here is the version of the geometric Satake equivalence we need in the sequel.
\begin{thm}\label{T: geom Sat}
There is a natural equivalence of tensor categories $\Sat: \on{Rep}({^C}G_{\bQ_\ell})\cong\Sat_{G,\ell}^T$
such that the composition with the cohomology functor $\on{H}^*(\Gr_{\bar{k}}, -): \Sat_{G,\ell}^T\to \sigma\Mod_{\bQ_\ell}$
is the restriction functor $\on{Rep}({^C}G_{\bQ_\ell})\to \sigma\Mod_{\bQ_\ell}$ 
induced by the inclusion $\sigma\mapsto (1,q,\sigma)\in \hat{G}\rtimes(\bG_m\times\langle\sigma\rangle)$.
\end{thm}

\begin{proof}
If $G$ is split (so ${^C}G=\hat{G}^T$), this was stated in \cite[Lemma 5.5.14]{Z17b}. So we obtain a natural equivalence 
$\on{Rep}(\hat{G}_{\bQ_\ell}^T)\cong \Sat_{G,\tilde{k},\ell}^T$
satisfying the desired properties as in the theorem with $\sigma$ replaced by $\sigma^r$. In addition, $\hat{G}^T$ is equipped with a pinning (induced from the pinning of $\hat{G}$ as described in \cite[Corollary 5.3.23]{Z17b}). 

Now for $(\mF,\ga)\in \Sat_{G,\ell}^T$, we have a canonical isomorphism $\on{H}^*(\Gr_{\bar{k}},\mF)\cong \on{H}^*(\Gr_{\bar{k}},\sigma^*\mF)\cong\on{H}^*(\Gr_{\bar{k}},\mF)$.
Using the formalism as in \cite[Lemma A.3]{RZ} and an argument similar to \cite[Proposition A.6]{RZ} (see also \cite[Lemma 5.5.5]{Z17b} and the paragraphs before it), we see that the above equivalence induces the pinned action of $\sigma$ on $\hat{G}_{\bQ_\ell}^T$, and the desired equivalence.
\end{proof}
\begin{rmk}\label{R: geo Sat}
(1) Without adding the $\bG_m$ factor, the Galois action of $\sigma$ on $\hat{G}$ obtained by the formalism \cite[Lemma A.3]{RZ} does not preserve the pinning. The semidirect product of $\Ga_k$ with $\hat{G}(\bQ_\ell)$ using this action was denoted by ${^L}G^{\on{geom}}$ in \emph{loc. cit.} (and later denoted by ${^L}G^{\on{geo}}$ in \cite[\S 5]{Ri} and in \cite[\S 5.5]{Z17b}). Theorem \ref{T: geom Sat} is compatible with \cite[\S 5]{Ri} (see also \cite[Theorem 5.5.12]{Z17b}), via restriction along the injective map ${^L}G^{\on{geo}}\to {^C}G(\bQ_\ell)$. 
\quash{
where a certain category $\on{Rep}^c({^L}G^{\on{geo}})$ of finite dimensional continuous representations of ${^L}G^{\on{geo}}$ was proved to be equivalent to  $\Sat_{G,\ell}$. Namely, the following commutative is commutative
\[
\begin{CD}
\on{Rep}({^C}\hat{G}_{\bQ_\ell})@>\cong>> \Sat_{G,\ell}^T\\
@VVV@VVV\\
\on{Rep}^c({^L}G^{\on{geo}})@>\cong>> \Sat_{G,\ell},
\end{CD}
\]
where the left vertical arrow is the natural restriction functor (which is fully faithful), and right vertical arrow is the full embedding via the Galois descent.}

(2) One of the consequences of the above theorem is as follows. For a $\sigma$-invariant dominant weight $\hat\mu$, there is a unique (up to isomorphism) irreducible representation $V_{\hat\mu}$ of ${^C}G_{\bQ_\ell}$ such that $V_{\hat\mu}|_{\hat{G}}$ is irreducible of highest weight $\hat\mu$ and that the action of $\bG_m\times\langle\sigma\rangle$ on the lowest weight line of $V_{\hat{\mu}}$ (w.r.t. $(\hat{G},\hat{B})$) is trivial. 
Namely, under the geometric Satake, $V_{\hat\mu}$ corresponds to $\IC_{\hat\mu}$ equipped with the natural $\Gal(\tilde k/k)$-equivariant structure. Of course, this fact is well-known.
\end{rmk}

We will need the following properties of the geometric Satake equivalence. 


First, let $T$ be the abstract Cartan of $G$. We can define the Satake category $\Sat_{T,\ell}^T$ as a subcategory of perverse sheaves on $\Gr_T$.
For a choice a Borel subgroup $B\subset G$ over $\mO$, we have the correspondence $\Gr_T\xleftarrow{r}\Gr_B\xrightarrow{i}\Gr_G$.
Recall that for $\hat\la\in\xch(\hat T)$, we have the point $\hat\la(\varpi)\in\Gr_T(\tilde k)$.
Let $\Gr_{B,\hat\la}=r^{-1}(\hat\la(\varpi))$ be the (geometrically) connected component of $\Gr_B$ given by $\hat\la$. We write $r_{\hat\la}$ (resp. $i_{\hat\la}$) be the restriction of $r$ (resp. $i$) to $\Gr_{B,\hat\la}$.
Then we define the Mirkovi\'c-Vilonen's weight functor 
\[
    \on{CT}: \Sat_{G,\ell}^T\to \Sat_{T,\ell}^T,\quad  \on{CT}(\mF)=\bigoplus_{\hat\la}  r_{\hat\la,!} i_{\hat\la}^*\mF[(2\rho,\hat\la)].
\]
We note that $\on{CT}(\mF)$ is a sheaf on $\Gr_{T,\tilde k}$ naturally equipped with a $\Gal(\tilde k/k)$-equivariant structure so that the above definition makes sense. Then the the geometric Satake fits into the following commutative diagram
\begin{equation}\label{E: res levi}
\xymatrix{
\Rep({^C}G_{\bQ_\ell}) \ar^-{\Sat}[r]\ar_{\Res}[d]& \Sat_{G,\ell}^T \ar^{\on{CT}}[d]\\
\Rep( {^C}T_{\bQ_\ell}) \ar^-{\Sat}[r] &\Sat^T_{T,\ell}.
}
\end{equation}
This follows from the usual compatibility between the geometric Satake and the restriction to the maximal torus, with the Galois action taking into account.

For $\mF\in \Sat_{G,\ell}^T$, its Frobenius trace function 
$$f_{\mF}\in C_c(K\backslash G(F)/K,\bQ_\ell)$$ 
makes sense as usual. 
The next fact we need is a description of $f_{\Sat(V)}$ for $V\in {^C}G_{\bQ_\ell}$. For this purpose, we need to recall the so-called Brylinski-Kostant filtration. Let $\{\hat e,2\rho, \hat f\}$ be the principal $\fraks\frakl_2$-triple of $\hat{G}$ containing $\hat e\in \Lie \hat{B}$ and $2\rho\in \xcoch(\hat{T})\subset\Lie\hat{T}$. For a representation $V$ of $\hat{G}$ and $\hat\la\in \xch(\hat T)^{-}$, we define the Brylinski-Kostant filtration on $V(\hat\la)$ as
\[
   F_iV(\hat\la):=\ker(\hat f^{i+1}: V\to V)\cap V(\hat\la).
\]
Let $\on{gr}^F_\bullet V(\hat\la)$ denote its associated graded.
Note that if $\hat\la$ is $\sigma$-invariant, then $F_iV(\hat\la)$ is $\sigma$-stable and therefore $\sigma$ acts on $\on{gr}^F_\bullet V(\hat\la)$.

\begin{prop}\label{P: trace Frob}
Let $V$ be a representation of $^{C}G_{\bQ_\ell}$, and $\hat\la\in \xch(\hat T)^{\sigma,-}$. Then 
\[
   f_{\Sat(V)}(\hat\la(\varpi))=(-1)^{\langle2\rho,\hat\la\rangle}\tr((1,q,\sigma)\mid \on{gr}_i^FV(\hat\la))q^{-i}.
\]   
\end{prop}
\begin{proof}
This follows \cite[\S 5]{Z15}, by taking into account of the Galois action. Note that the proof of Lemma 5.8 of \emph{loc. cit.} relies on the existence of ``big open cell" of the affine Grassmannian, which is not known in mixed characteristic. However, one can easily replace the purity argument in \emph{loc. cit.} by a parity argument, e.g. the argument in the middle of p. 452 of \cite{Z17a}.
\end{proof}
\begin{rmk}
Recall that for the representation $V=V_{\hat\mu}$ of ${^C}G_{\bQ_\ell}$ as in Remark \ref{R: geo Sat}, and for $\hat\la\in\xch(\hat T)^{\sigma}$, Jantzen's twisted character formula  (\cite[Satz 9]{Ja}) expresses $\tr((1,1,\sigma)\mid V_{\hat\mu}(\hat\la))$ as the dimension of a representation of a reductive group $\hat G_\sigma$ whose weight lattice is $\xch(\hat T)^\sigma$. 
It would be very interesting to have its $q$-analogue, expressing $\sum_{i} \tr(\sigma\mid \on{gr}_i^FV_{\hat\mu}(\hat\la))q^{-i}$ in terms of representations of $\hat G_\sigma$. 
\end{rmk}
We do not have such a formula at the moment. The following lemma is sufficient for our purpose.
\begin{lem}\label{L: integrality trace}
Let $V_{\hat\mu}$ be as above. Then for every $\hat\la\in \xch(\hat T)^\sigma$, $\tr(\sigma\mid \on{gr}_{i}^FV_{\hat\mu}(\hat\la))\in\bZ$.
\end{lem}
\begin{proof}
We may assume that $\hat\la\in\xch(T)^{\sigma,-}$. The root datum of $G$ defines a reductive group $\bG$ over $\bC$ equipped with a $\bC$-automorphism $\sigma$. Let $\Gr_{\mathbb G}$ be its affine Grassmannian, acted by $\sigma$. We have the geometric Satake for $\Gr_{\bG}$, and the analogous statement of Proposition \ref{P: trace Frob} in this setting is
\[\tr(\sigma \mid \on{gr}_i^FV_{\hat\mu}(\hat\la))= \tr(\sigma\mid \mH_{\hat\la(\varpi)}^{-2i+\langle 2\rho,\hat\la\rangle}\IC_{\hat\mu}),\]
where $\mH_{\hat\la(\varpi)}^{-2i+\langle 2\rho,\hat\la\rangle}\IC_{\hat\mu}$ denotes the stalk cohomology of $\IC_{\hat\mu}$ at $\hat\la(\varpi)$.
Over $\bC$, the sheaf $\IC_{\hat\mu}$ has a natural $\bZ$-structure preserved by the action of $\sigma$ (\cite[Proposition 8.1]{MV}). The lemma follows.
\end{proof}

\subsection{The representation ring}
We generalize some well-known relations between the representation ring of a reductive group and its ring of invariant functions. Let $E$ be a characteristic zero field.
Let $\on{Rep}^+(\hat{G}^T_E)\subset \on{Rep}(\hat{G}^T_E)$ denote the subcategory consisting of those objects on which all weights of $\bG_m\subset \hat{G}\rtimes_{\Ad \rho_\ad} \bG_m=\hat G^T$ are $\geq 0$. Let $\on{Rep}(V_{\hat{G},\rho_\ad,E})$ denote the category of finite dimensional algebraic representations of $V_{\hat{G},\rho_\ad,E}$.

\begin{lem}
The inclusion $ \hat{G}^T\subset V_{\hat{G},\rho_\ad}$ induces an equivalence of categories $\on{Rep}^+(\hat{G}^T_E)\cong \on{Rep}(V_{\hat{G},\rho_\ad,E})$.
\end{lem}
\begin{proof}
First, under the inclusion $\hat{G}\times^{Z_{\hat G}}\hat{T}\to V_{\hat{G}}$, an irreducible representation $V$ of $\hat{G}_E\times \hat{T}_E$ can be extended to a representation of $V_{\hat{G},E}$ if and only if the following holds:
if $V|_{\hat{G}_E}$ has the highest weight $\hat\mu$, and $V|_{\hat{T}_E}$ has the weight $\hat\nu$, then $\hat\nu+w_0(\hat\mu)$ is a sum of nonnegative roots of $\hat{G}$. Therefore, an irreducible representation of $(\hat{G}\times\bG_m)/(2\rho\times\id)(\mu_2)$ can be extended to a representation of $V_{\hat{G},\rho_\ad}$ if and only if the following two conditions hold: 
\begin{enumerate}
\item[(i)] $\bG_m$ acts on $V$ by some weight $n$; and 
\item[(ii)] If $\hat{\mu}$ is the highest weight of $V$ as a $\hat{G}$-representation, then $\langle2\rho,\hat{\mu}\rangle \leq n$ and $(-1)^{\langle 2\rho,\hat{\mu}\rangle}=(-1)^{n}$.
\end{enumerate}
But under the isomorphism \eqref{R: another version}, this exactly corresponds to an object in $\on{Rep}^+(\hat{G}^T_E)$.
\end{proof}


Now let $\sigma$ be a finite order automorphism of $(\hat G,\hat B,\hat T, \hat e)$ as before. We let $\on{Rep}^{\on{int}}(\bG_m\times\langle\sigma\rangle)$ denote the full tensor subcategory of finite dimensional algebraic $E$-linear representations of $\bG_m\times\langle\sigma\rangle$ consisting of $W$ satisfying the following two properties
\begin{itemize}
\item all $\bG_m$-weights in $W$ are $\geq 0$;
\item the trace of $\sigma$ on each $\bG_m$-weight subspace of $W$ is an integer.
\end{itemize}
This is equivalent to requiring that the trace of $(q,\sigma)$ on $W$ is an integer.

Let $\on{Rep}^{\on{int}}({^C}G_E)$ be the full tensor subcategory of $\on{Rep}({^C}G_E)$ consisting of those representations $V$ such that for every $\sigma$-invariant weight $\hat \la$ of $\hat{T}$ and every $i\geq 0$, the space $\on{gr}_i^FV(\hat \la)$ as a representation of $\bG_m\times \langle \sigma\rangle\subset \hat{G}\rtimes(\bG_m\times\langle\sigma\rangle)$  belongs to $\on{Rep}^{\on{int}}(\bG_m\times\langle\sigma\rangle)$. Note that the restriction of such a representation to $\hat{G}^T$ belongs to $\on{Rep}^+(\hat{G}^T_E)$, and therefore we have the functor 
\[
  \on{Rep}^{\on{int}}({^C}G_E)\to \on{Rep}(V_{\hat G,\rho_\ad,E}\rtimes\langle\sigma\rangle). 
\]
When $\sigma$ is trivial, we have $\on{Rep}^{\on{int}}({^C}G_E)\cong \on{Rep}^+(\hat{G}^T_E)\cong \on{Rep}(V_{\hat G,\rho_\ad,E})$.
\begin{lem}\label{L: K-ring to function}
There exists a unique homomorphism $\tr:K(\on{Rep}^{\on{int}}({^C}G_E))\to\bZ[V_{\hat{G},\rho_\ad}|_{d_{\rho_\ad}=q}]^{c_\sigma(\hat{G})}$ making the following diagram commutative
\begin{equation}\label{E: Kgroup to function}
\xymatrix{
K(\on{Rep}^{\on{int}}({^C}G_E))\ar^-{\tr}[rrrr] \ar_-{K(\Res)}[d]&&&&\bZ[V_{\hat{G},\rho_\ad}|_{d_{\rho_\ad}=q}]^{c_\sigma(\hat{G})}\ar^{\eqref{eq: inj inv function on VT}\circ\Res}[d]&\\
K(\on{Rep}^{\on{int}}({^C}T_E)) \ar^-{[V]\mapsto  \sum_{\hat\la\in \xch(\hat T)^\sigma} \tr((q,\sigma)| V(\hat\la))e^{\hat\la}}[rrrr] &&&& \bZ[\xch(\hat T)^\sigma].
}
\end{equation}
\end{lem}
\begin{proof}
We define the map
\[
  \tr: K(\on{Rep}^{\on{int}}({^C}G_E))\to K(\on{Rep}(V_{\hat G,\rho_\ad,E}\rtimes \langle\sigma\rangle)\to \overline{E}[V_{\hat{G},\rho_\ad}|_{d_{\rho_\ad}=q}]^{c_\sigma(G)},
\] 
where the second arrow is induced by taking the trace of representations at elements in $V_{\hat{G}}|_{d=\rho_\ad(q)}\sigma$. The diagram is clearly commutative when the right vertical map is tensored with $\overline{E}$. To see that $\tr$ factors through $\bZ[V_{\hat{G},\rho_\ad}|_{d=q}]^{c_\sigma(\hat{G})}$, it is enough to note that $\bZ[V_{\hat{G},\rho_\ad}|_{d=q}]^{c_\sigma(\hat{G})}= \overline{E}[V_{\hat{G},\rho_\ad}|_{d=q}]^{c_\sigma(\hat{G})}\cap \bZ[\xch(\hat T)^\sigma]$  inside $\overline{E}[\xch(\hat T)^\sigma]$ which follows easily from
Lemma \ref{L: image Chevalley restriction}. The uniqueness is clear as the right vertical map is injective.
\end{proof}

\begin{lem}\label{L: int Vmu}
The homomorphism in Lemma \ref{L: K-ring to function} is surjective.
\end{lem}
\begin{proof}
For a $\sigma$-invariant dominant weight $\hat\mu$, let $V_{\hat\mu}$ be the unique (up to isomorphism) irreducible representation of ${^C}G_E$ as in Remark \ref{R: geo Sat} (2). We claim that $V_{\hat\mu}\in \on{Rep}^{\on{int}}({^C}G_E)$. Then the lemma follows from the explicit description of $\bZ[V_{\hat{G},\rho_\ad}|_{d_{\rho_\ad}=q}]^{c_\sigma(\hat{G})}$ inside $\bZ[\xch(\hat T)^\sigma]$, as in Lemma \ref{L: image Chevalley restriction}.

We need to show that $\tr((1,q,\sigma)\mid \on{gr}_F^iV_{\hat\mu}(\hat\la))\in \bZ$ for $\sigma$-invariant weight $\hat\la$. 
Note that under the map
$\hat{G}\times \bG_m\to \hat{G}^T$ from the isomorphism \eqref{R: another version}, the second factor $\bG_m$ acts on $V_{\hat\mu}$ by a fixed weight (namely $\langle2\rho,\hat\mu\rangle$). Therefore, $(1,q)\in \hat{G}\rtimes\bG_m=\hat{G}^T$ acts on $V_{\hat\mu}(\hat\la)$ by $q^{\langle\rho_{\ad}, \hat\la-w_0\hat\mu\rangle}$. Therefore, the lemma follows by noticing that the trace of $(1, 1,\sigma)$ on $V_{\hat\mu}(\hat\la)$ is an integer, by Lemma \ref{L: integrality trace}.
\end{proof}

\subsection{From $\Sat$ to $\Sat^\cl$}
Let $\Sat_{G,\ell}^{T,\on{int}}$ be the full subcategory of $\Sat_{G,\ell}^{T}$ consisting of those $\mF$ such that $f_{\mF}\in H_G$, i.e. $f_{\mF}(\hat\la(\varpi))\in\bZ$ for every $\hat\la\in\xch(\hat T)^\sigma$. 
By Proposition \ref{P: trace Frob}, we have the following statement.
\begin{lem}
The geometric Satake induces an equivalence $\on{Rep}^{\on{int}}({^C}G_E)\cong \on{Sat}_{G,\ell}^{T,\on{int}}$.
\end{lem}

\begin{lem}\label{L: surj fun}
Taking the trace Frobenius function induces a surjective ring homomorphism 
\[\tr: K(\on{Sat}_{G,\ell}^{T,\on{int}})\to H_G.\]
\end{lem}
\begin{proof}
For dominant $\hat\mu\in\xch(\hat T)^{\sigma}$, by Proposition \ref{P: trace Frob} and Lemma \ref{L: int Vmu},
\begin{equation}\label{E: tr function}
   f_{\IC_{\hat \mu}}=(-1)^{\langle2\rho,\hat\mu\rangle}1_{K\hat\mu(\varpi) K}+ \sum_{\hat\mu'<\hat\mu}a_{\hat\mu\hat\mu'}1_{K\hat\mu'(\varpi) K},\quad a_{\hat\mu\hat\mu'}\in\bZ,
\end{equation}
where $1_{K\hat\mu'(\varpi) K}$ denotes the characteristic function on $K\hat\mu'(\varpi)K$ and ``$<$" denotes the Bruhat order. Then the lemma follows since these $1_{K\hat\mu(\varpi) K}$'s form a $\bZ$-basis of $H_G$.
\end{proof}

Note that the Grothendieck-Lefschetz trace formula implies that
\[
   \tr\circ K(\on{CT})= \on{CT}^{\cl}\circ \tr: K(\Sat_{G,\ell}^{T,\on{int}})\to \bZ[\xch(\hat{T})^\sigma].
\]   
Together with \eqref{E: res levi} and \eqref{E: Kgroup to function}, we obtain the following commutative diagram
\[\xymatrix{
K(\on{Rep}^{\on{int}}({^C}G_{\bQ_\ell}))\ar_-{K(\Sat)}^-\cong[d]\ar^{\tr}[r]& \bZ[V_{\hat G}|_{d=\rho_\ad(q)}]^{c_\sigma(G)}\ar^\Res[r] &  \bZ[V_{\hat T}|_{d=\rho_\ad(q)}]^{c_\sigma(\hat{N}_0)}\ar^{\eqref{eq: inj inv function on VT}}[d]\\
K(\Sat_{G,\ell}^{T,\on{int}})\ar^-{\tr}[r]& H_G \ar^{\on{CT}^{\cl}}[r] & \bZ[\xch(\hat T)^\sigma].
}\]
Both trace maps $\tr$ in the above diagram are surjective, by Lemma \ref{L: int Vmu} and \ref{L: surj fun}. Therefore, we can complete the diagram by adding an isomorphism $\Sat^\cl$ in the middle column, as desired.

\quash{
\begin{lem}
Let $(H, B_H, T_H, e_H)$ be a pinned reductive group over a field $E$ of characteristic zero and let $\sigma$ be a finite order automorphism of the pinning. If $\nu\in \xch(T_H)$ is a $\sigma$-invariant dominant weight, and let $V_\nu$ be the corresponding irreducible representation of $H$. Then
\begin{enumerate}
\item The representation $V_\nu$ extends uniquely to a representation of $H\rtimes\langle\sigma\rangle$ such that $\sigma=\id: V_\nu(\nu)\to V_\nu(\nu)$.
\item For every $\sigma$ invariant weight $\la$, the trace of $\sigma: V_\nu(\la)\to V_\nu(\la)$ is an integer.
\end{enumerate}
\end{lem}
\begin{proof}
Part (2) follows from Jantzen's character formula
\end{proof}
}

\end{document}